\documentclass[12pt,a4paper,reqno]{amsart} 

\usepackage{amsfonts}
\usepackage{amssymb}
\usepackage{amsmath}
\usepackage{amsthm}
\usepackage{cite}
\usepackage{fancyhdr}
\usepackage[colorlinks,linkcolor=blue]{hyperref}

\theoremstyle{plain}
\newtheorem{Thm}{Theorem}[section]
\newtheorem{Pro}[Thm]{Proposition}
\newtheorem{Def}{Definition}[section]
\newtheorem{Rk}[Thm]{Remark}
\newtheorem{Lem}[Thm]{Lemma}

\newtheorem{Ex}{Example}[section]

\numberwithin{equation}{section}
\begin{document}
\title[fully nonlinear elliptic equations]{On the exterior Dirichlet problem for a class of fully nonlinear elliptic equations} 

\author[T.Y. Jiang]{Tangyu Jiang}
\address[T.Y. Jiang]{School of Mathematical Sciences\\
Beijing Normal University \\
	Beijing, 100875\\
	P.R. China}
\email{{\tt 201821130025@mail.bnu.edu.cn}}

\author[H.G. Li]{Haigang Li}
\address[H.G. Li]{School of Mathematical Sciences\\
Beijing Normal University \\
	Beijing, 100875\\
	P.R. China}
\email{{\tt hgli@bnu.edu.cn}}

\author[X.L. Li]{Xiaoliang Li$^*$}
\address[X.L. Li]{School of Mathematical Sciences\\
Beijing Normal University \\
	Beijing, 100875\\
	P.R. China}
\email{{\tt rubiklixiaoliang@163.com}}

\thanks{H.G. Li was supported by NSFC (11631002, 11971061).}

\thanks{$^*$ Corresponding author: Xiaoliang Li}


\begin{abstract}
In this paper, we mainly establish the existence and uniqueness theorem for solutions of the exterior Dirichlet problem for a class of fully nonlinear second-order elliptic equations related to the eigenvalues of the Hessian, with prescribed generalized symmetric asymptotic behavior at infinity. Moreover, we give some new results for the Hessian equations, Hessian quotient equations and the special Lagrangian equations, which have been studied previously. 
\end{abstract}
\maketitle
\section{Introduction}\label{sec:intro}
In this paper, we will study the exterior Dirichlet problem for the following fully nonlinear, second-order partial differential equations of the form,
\begin{equation}\label{eq:pro}
\left\{
\begin{array}{ll}
f(\lambda(D^2u))=1, & \mathrm{in}\ \mathbb{R}^n\setminus \overline{D},\\
u=\varphi, & \mathrm{on}\ \partial D,
\end{array}
\right.
\end{equation}
where $D$ is a bounded open set in $\mathbb{R}^n$ $(n\geq 3)$,  $\lambda(D^2 u)=\{\lambda_i(D^2u)\}_{i=1}^n$ denotes the eigenvalues of the Hessian matrix $D^2u$, and $f$ is a smooth symmetric function defined in an open convex cone $\Gamma\subsetneqq \mathbb{R}^n$, with vertex at the origin and a location such that
$$\{\lambda\in\mathbb{R}^n|\lambda_i> 0, i=1,\cdots, n\}\subset\Gamma\subset \{\lambda\in\mathbb{R}^n|\sum_{i=1}^n\lambda_i>0\}.$$ As typical examples embraced in \cite{Caffarelli1984, Caffarelli1985,Trudinger1995}, such $f$ may include the $k$-th elementary symmetric function
\begin{equation}\label{eq:k-sigma}
\sigma_k(\lambda):=\sum_{1\leq i_1<\cdots<i_k\leq n}\lambda_{i_1}\cdots\lambda_{i_k},
\end{equation}
the quotient function $(1\leq l<k\leq n)$
\begin{equation}\label{eq:quotient}
S_{k,l}(\lambda):=\frac{\sigma_k(\lambda)}{\sigma_l(\lambda)},
\end{equation}
and the special Lagrangian operator
\begin{equation}\label{eq:Lag}
\sum_{i=1}^n\arctan \lambda_i,
\end{equation}
with which \eqref{eq:pro} corresponds to the $k$-Hessian equations (particularly, the Poisson equation $\Delta u=1$ if $k=1$ and Monge-Amp\`ere equation $\det(D^2u)=1$ if $k=n$), the Hessian quotient equations and the special Lagrangian equations, respectively.

For equation \eqref{eq:pro}, a special class of fully nonlinear equations of Hessian type, there has been many
excellent results regarding the existence of solutions to various interior boundary value problems. Caffarelli, Nirenberg and Spruck \cite{Caffarelli1984, Caffarelli1985} and Trudinger \cite{Trudinger1995} solved the classical Dirichlet problem by the method of continuity, under the structural conditions that $f$ is a concave function and satisfies
\begin{equation}\label{eq:increase}
\frac{\partial f}{\partial \lambda_i}>0 \  in\ \Gamma, \ i=1,\cdots,n,
\end{equation} 
where the equation is elliptic and the cases \eqref{eq:k-sigma}-\eqref{eq:Lag} are included. Then in the setting of Riemannian manifolds, Guan \cite{Guan1999} treated a more general elliptic equation and stated that a smooth solution of the Dirichlet problem exists if an associated subsolution exists, without any geometric restrictions to the boundary. For more related studies, we refer the reader to \cite{Trudinger1990,Urbas1990} for the Dirichlet problem and \cite{Lions-Trudinger-Urbas-1986,Trudinger1987,Ma-Qiu-2019,Chen-Ma-Wei-2019,Urbas1995, Liberman2013} for the Neumann problem and the nonlinear oblique boundary value problem, and the references therein.

Turning to the counterpart, the study of the exterior problem \eqref{eq:pro} in an unbounded domain also received increasing attention in recent years. As we know, a celebrated J\"orgens-Calabi-Pogorelov theorem \cite{Jorgens1954, Calabi1958, Pogorelov1972, Cheng-Yau-1986, Caffarelli1995,Jost2001} states that any classical convex solution of 
\begin{equation}\label{eq:MA}
\det(D^2u)=1, 
\end{equation}
in $\mathbb{R}^n (n\geq 3)$ must be a quadratic polynomial. On exterior domains, Caffarelli and Li \cite{Caffarelli-Li-2003} showed that any convex viscosity
solution of \eqref{eq:MA} will approach uniformly, with order of the fundamental solution of the Laplacian, to a quadratic polynomial at infinity, i.e., there holds
\begin{equation}
\label{eq:C-Li}
\limsup_{|x|\to\infty}\left(|x|^{n-2}\left|u(x)-\left(\frac12x^TAx+b\cdot x+c\right)\right|\right)<\infty,
\end{equation}
where $A$ is some symmetric positive definite matrix with $\det A=1$. Moreover, for every such matrix $A$, 
Li and Lu \cite{Li-Lu-2018} completed the
characterization of solvability to the exterior Dirichlet problem for  \eqref{eq:MA} in terms of the prescribed asymptotic behavior \eqref{eq:C-Li}. In the same spirit, the existence theorem of exterior problem \eqref{eq:pro} has established recently by Perron's method for $k$-Hessian equations $(2\leq k\leq n)$ in \cite{Bao-Li-Li-2014}, for Hessian quotient equations in \cite{Li-Li-Zhao-2019, Li-Li-2018}, and for the special Lagrangian equations in \cite{Bao-Li-2013,Li2019}, respectively, with the similar settings as \eqref{eq:C-Li} where the matrices $A$ are choosed to adapt the corresponding equation. But for the general equation \eqref{eq:pro} of abstract form, few results are known on exterior domains. In \cite{Li-Bao-2014}, assuming \eqref{eq:increase} and 
that there is a positive constant $a^*$ such that
\begin{equation}\label{eq:c}
f(a^*(1,1,\cdots,1))=1,
\end{equation}
Li and Bao proved
the existence of viscosity solutions of problem \eqref{eq:pro} prescribed at infinity by \eqref{eq:C-Li} where the matrix $A$ is fixed to be $a^*I$. 
Except for these existence results, we refer to \cite{Bao-Li-Zhang-2015,Li-Li-Yuan-2019} for more information about Hessian type equations outside a bounded domain, where the Liouville type results were exploited.

This paper aims to establish the existence and uniqueness theorem for the exterior Dirichlet problem \eqref{eq:pro} with some prescribed asymptotic behavior at infinity. We only assume \eqref{eq:increase}, \eqref{eq:c} and another boundary condition, without requiring the concavity assumption. We would like to point out that here our hypotheses on function $f$ are so weak that it not only can cover the situation considered in \cite{Caffarelli1984,Caffarelli1985,Trudinger1995}   but also allows more examples of a fully nonlinear equation. By introducing the concept of generalized symmetric (an extended version of radially symmetric to be specified later) subsolutions and then applying Perron's method, it is proved here that not just for $a^* I$ introduced in \eqref{eq:c} but for more positive definite matrices satisfying equation \eqref{eq:pro}, problem \eqref{eq:pro} with prescribed asymptotics of the type similar to \eqref{eq:C-Li} admits a unique viscosity solution, thus breaking the constraint of radially symmetric on asymptotic functions imposed in \cite{Li-Bao-2014} and extending its result to be valid for  more general prescribed asymptotic behavior condition. This also extends those existence results obtained in \cite{Bao-Li-Li-2014, Li-Li-Zhao-2019,Li-Li-2018,Li2019} particularly for cases \eqref{eq:k-sigma}-\eqref{eq:Lag} to the general $f$ fulfilling \eqref{eq:increase}. Moreover, since we use the new technique presently to construct subsolutions, which analyzes the corresponding implicit ordinary differential equation and does not rely on the homogeneity of $f$, our theorem exhibits some new related results for those particular equations. Before giving the precise statement, we here introduce some notations and definitions.

First, we say that a function $u\in C^2(\mathbb{R}^n\setminus\overline{D})$ is admissible if $\lambda(D^2u)\in\overline{\Gamma}$. Then we use $\mathrm{USC}(\Omega)$ and $\mathrm{LSC}(\Omega)$ to respectively denote the set of upper and lower semicontinuous real valued functions on $\Omega\subset\mathbb{R}^n$. The definition of viscosity solution to \eqref{eq:pro} follows from \cite{Ishii1992}.
\begin{Def}
A function $u\in \mathrm{USC}(\mathbb{R}^n\setminus\overline{D})$ $(\mathrm{LSC}(\mathbb{R}^n\setminus\overline{D}))$ is said to be a viscosity subsolution (supersolution) of equation \eqref{eq:pro}, if for any admissible function $\psi$ and point $\bar{x}\in \mathbb{R}^n\setminus\overline{D}$  satisfying $$\psi(\bar{x})=u(\bar{x})\quad\mathrm{and}\quad \psi\geq(\leq) u\ on\ \mathbb{R}^n\setminus\overline{D}$$ we have $$f(D^2\psi(\bar{x}))\geq(\leq)1.$$
A function $u\in C^0(\mathbb{R}^n\setminus\overline{D})$ is said to be a viscosity solution of equation \eqref{eq:pro} if it is both a viscosity subsolution and supersolution.
\end{Def}

\begin{Def}
Let $\varphi\in C^0(\partial D)$. A function $u \in \mathrm{USC}(\mathbb{R}^n\setminus D)$ $(u \in \mathrm{LSC}(\mathbb{R}^n\setminus D))$ is said to be a viscosity subsolution (supersolution) of problem \eqref{eq:pro}, if $u$ is a viscosity subsolution (supersolution) of equation \eqref{eq:pro} and $u\leq (\geq)\varphi$ on $\partial D$. A function $u\in C^0(\mathbb{R}^n\setminus D)$ is said to be a viscosity solution of problem \eqref{eq:pro} if it is both a subsolution and a supersolution.
\end{Def}

Following the notations in \cite{Bao-Li-Li-2014,Li-Li-Zhao-2019}, we define the concept of generalized symmetric (abbreviated to G-Sym) solution of \eqref{eq:pro} as follows:
\begin{Def}
For a symmetric matrix $A$, we call $u$ a G-Sym function with respect to $A$ if it is a function of $s=\frac12x^TAx$, $x\in\mathbb{R}^n$, that is, $$u(x)=u(s):=u(\frac12x^TAx).$$ If $u$ is a solution of equation \eqref{eq:pro} and is also a G-Sym function with respect to some symmetric matrix, we say that $u$ is a G-Sym solution of \eqref{eq:pro}.
\end{Def}
 
Next, we assume that the symmetric function $f$ in \eqref{eq:pro} satisfies
\begin{equation}\label{eq:boundary}
\limsup_{\lambda\to\lambda_0}f(\lambda)<1, \quad \text{for every}\ \lambda_0\in\partial\Gamma. 
\end{equation}
Then for a symmetric positive definite matrix $A$, let $$a:=\lambda(A)=(a_1,a_2,\cdots,a_n)$$ denotes its eigenvalues and define 
$$\hat{a}=\max_{1\leq i\leq n}a_i,\quad\hat{f}_\lambda(a)=\max_{1\leq i\leq n}\frac{\partial f}{\partial\lambda_i}(a).$$ 
Set $$\mathcal{A}:=\left\{A\in S^+(n):f(a)=1,\  \frac{\nabla f(a)\cdot a}{2\hat{a}\hat{f}_\lambda(a)}>1\right\}$$
where $S^+(n)$ is the set containing all of real $n\times n$ symmetric positive definite matrix. It is easy to verify that $a^*I\in\mathcal{A}$ for $a^*$ defined by \eqref{eq:c}. 
 
 Our main result is:
\begin{Thm}\label{thm:main}
Let $D$ be a smooth, bounded, strictly convex open subset of $\mathbb{R}^n$, $n\geq 3$ and let $\varphi\in C^2(\partial D)$. Assume that $f$ satisfies \eqref{eq:increase}, \eqref{eq:c} and \eqref{eq:boundary}. Then for any given $A\in\mathcal{A}$ and $b\in\mathbb{R}^n$, and for each $\delta>0$ such that 
\begin{equation}\label{eq:alpha-d}
\alpha_\delta:=\frac{\nabla f(a)\cdot a}{(2\hat{a}+\delta)\hat{f}_\lambda(a)}>1,
\end{equation}
 there exists a constant $c_*$ depending on $n,b,A,D,f,\delta$ and $\|\varphi\|_{C^2(\partial D)}$, such that for every $c>c_*$ there exists a unique viscosity solution $u\in C^0(\mathbb{R}^n\setminus D)$ of \eqref{eq:pro} satisfying
\begin{equation}\label{eq:asym-s}
\limsup_{|x|\to\infty}\left(|x|^{2\alpha_\delta-2}\Big{|}u(x)-(\frac12x^TAx+b\cdot x+c)\Big{|}\right)<\infty.
\end{equation}  

In particular, if $A=a^*I$ with $a^*$ defined in \eqref{eq:c}, or the $A\in\mathcal{A}$ with the property that $\frac{\partial f}{\partial\lambda_{i_0}}(a)>\max_{i\neq i_0}\frac{\partial f}{\partial\lambda_i}(a)$ for some $1\leq i_0\leq n$, then the above assertion holds further for $\delta\geq 0$ whenever $\alpha_\delta>1$.
\end{Thm}

\begin{Rk}\label{rk:cover}
When $A=a^*I$, Theorem \ref{thm:main} holds for any $0\leq\delta<(n-2)a^*$, which covers the result in \cite{Li-Bao-2014} where the particular case $\delta=0$ was studied (note $\alpha_0=\frac{n}{2}$ and \eqref{eq:asym-s} becomes \eqref{eq:C-Li} in this case). 
\end{Rk}

\begin{Rk}\label{rk:particular}
When $f$ is in the special forms \eqref{eq:k-sigma}-\eqref{eq:Lag} respectively, Theorem \ref{thm:main} can apply to the $k$-Hessian equations $(1\leq k\leq n)$, Hessian quotient equations and the special Lagrangian equations. Namely, those typical functions fulfill the conditions \eqref{eq:increase}, \eqref{eq:c} and \eqref{eq:boundary}, and for a rigorous verification we refer to \cite{Caffarelli1985} for \eqref{eq:k-sigma} and \eqref{eq:Lag}, and \cite{Trudinger1995} for \eqref{eq:quotient}. Notice that the solvability of problem \eqref{eq:pro} for these particular equations has studied separately in \cite{Bao-Li-Li-2014,Li-Li-Zhao-2019,Li-Li-2018, Caffarelli-Li-2003, Li2019}, while we obtain the new results for them in a different setting about prescribed asymptotic behavior at infinity,  complementing the previous related works. We would like to illustrate that more precisely in Subsection \ref{subsec:3-2} for each special equation by providing some specific corresponding examples of elements in $\mathcal{A}$.



\end{Rk}

The proof of Theorem \ref{thm:main} consists of constructing a family of G-Sym subsolutions of equation \eqref{eq:pro} with proper uniformly asymptotic behavior at infinity and carrying out the Perron's process. Unlike the kinds of literature \cite{Bao-Li-Li-2014,Li-Li-Zhao-2019,Li-Li-2018, Li2019,Caffarelli-Li-2003} where certain operators \eqref{eq:k-sigma}-\eqref{eq:Lag} would accurately act on G-Sym functions due to the homogeneity of $\sigma_k$ operators, or the paper \cite{Li-Bao-2014} seeking the radial solution to \eqref{eq:pro} whose eigenvalues of the Hessian matrix can be computed explicitly, one cannot directly represent the eigenvalues of the Hessian to the general G-Sym function in a precise way. This causes the major difficulty of searching for G-Sym subsolutions of \eqref{eq:pro} with $f$ being of abstract form. We tackle it by invoking a crucial estimate for eigenvalues of a symmetric matrix (see Theorem \ref{thm:Weyl}) and introducing a suitable nonnegative parameter $\delta$, which allows us to compare $f$ between imprecise $\lambda(D^2u)$ and a precise point in $\Gamma$ related to the $\delta$. 
Then we obtain the desired G-Sym subsolutions of \eqref{eq:pro} with respect to $A\in\mathcal{A}$ by solving an implicit ordinary differential equation. Our approach is purely inhomogeneous, and we take inspiration from \cite{Li-Bao-2014} where the radial case was studied. However, in the G-Sym case, some different technical difficulties need to be dealt with due to the non-equivalence of eigenvalues of $A$. Additionally, it should be noticed that here we need, as handled in \cite{Bao-Li-Li-2014,Li-Li-Zhao-2019}, to prove Theorem \ref{thm:main} when $A$ is diagonal and $b=0$ by the invariance of $f$ under the action of Euclidean group; but we cannot further assume $A=I$ unless $f$ is conformal invariant (the Monge-Amp\`ere equation for example).

The remainder of this paper is organized as follows. In Section \ref{sec:2}, we will construct appropriate subsolutions of equation \eqref{eq:pro}, which are G-Sym functions and possess certain uniformly asymptotic behavior at infinity. Then we prove Theorem \ref{thm:main} by Perron's method in Subsection \ref{subsec:3-1}. At last, in Subsection \ref{subsec:3-2}, we state some applications of Theorem \ref{thm:main} for problem \ref{eq:pro} associated with specific forms \eqref{eq:k-sigma}-\eqref{eq:Lag}, and further present respective examples to compare our result  and those built in previous works \cite{Bao-Li-Li-2014,Li-Li-Zhao-2019,Li-Li-2018, Caffarelli-Li-2003, Li2019}.  


Throughout the rest sections, we shall always use the symbol $$\lambda(A):=(\lambda_1(A),\lambda_2(A),\cdots,\lambda_n(A))$$ with the order $\lambda_1(A)\leq\lambda_2(A)\leq\cdots\leq\lambda_n(A)$ to denote the eigenvalues of a real $n\times n$ symmetric matrix $A$.



\section{Generalized symmetric subsolutions}\label{sec:2}
This section is devoted to the construction of G-Sym subsolutions of equation \eqref{eq:pro}, equipped with  \eqref{eq:increase}, \eqref{eq:c} and \eqref{eq:boundary}. It is an essential step to apply the Perron's arguments to prove Theorem \ref{thm:main} in the next section. Via a delicate process involving estimating the eigenvalues of the Hessian for G-Sym functions and analyzing an implicit ODE, we ultimately obtain G-Sym subsolutions with respect to the diagonal matrix of the set $\mathcal{A}$ and determine their asymptotic behavior at infinity, see Proposition \ref{pro:sub-f}.

We start with some preparations. For $A =\text{diag}(a_1, a_2,\cdots, a_n)$ with $0<a_1\leq a_2\leq\cdots\leq a_n$, suppose throughout this section that $u:=u(s)$ is a G-Sym function with respect to $A$, where $s=\frac12x^TAx=\frac12\sum_{i=1}^na_ix_i^2$, $x\in\mathbb{R}^n$. Then 
\begin{equation}\label{eq:D2u}
(D^2u)_{ij}=a_i\delta_{ij}u'+a_ia_jx_ix_ju''.\end{equation}
To make $u$ be a subsolution of \eqref{eq:pro}, i.e. $f(\lambda(D^2u))\geq 1$, we need acquire the information of the eigenvalues $\lambda_i(D^2u)$, $i=1,\cdots,n$. Since the uncertainty of $a_i$, it seems impossible to directly compute $\lambda_i(D^2u)$. So we turn to seek an available estimate for them. Indeed, there holds:
\begin{Lem}\label{lem:D2u}
Assume $u'(s)>0$ and $u''(s)\leq 0$ for $s>0$. Then for $i=1,\cdots,n$, 
\begin{equation*}
a_iu'(s)+\sum_{j=1}^na_j^2x_j^2u''(s)\leq \lambda_i(D^2u(s))\leq a_iu'(s).
\end{equation*}

In particular, when $a_1=a_2=\cdots=a_n=a$,  
\begin{gather*}
\lambda_1(D^2u(s))=au'(s)+2asu''(s); \\
 \lambda_i(D^2u(s))=au'(s), \ i=2,\cdots,n.
\end{gather*}
\end{Lem}

\begin{proof}
In \eqref{eq:D2u}, denote $$D^2u:=A_1+A_2$$ with $$A_1=\text{diag}(a_1u',a_2u',\cdots,a_nu')\quad \text{and}\quad A_2=(a_ia_jx_ix_ju'').$$ By a straightforward computation, one has
$$\lambda_1(A_2)=\sum_{j=1}^n a_j^2 x_j^2u'';\quad \lambda_{i}(A_2)=0,\ i=2,\cdots,n.$$
We next use $\lambda(A_1)$ and $\lambda(A_2)$ to estimate $\lambda(A_1+A_2)$, by employing a key assertion, stated below as a simple version, whose complete content and proof can be found in \cite{Horn1985}. 
\begin{Thm}[Weyl \cite{Horn1985}]\label{thm:Weyl}
Let $A$ and $B$ be real $n\times n$ symmetric matrices. Then
\begin{equation}\label{eq:Wue}
\lambda_i(A+B)\leq\lambda_{i+j}(A)+\lambda_{n-j}(B), \quad j=0,1,\cdots,n-i
\end{equation}
for each $i=1,\cdots,n$. Also,
\begin{equation}\label{eq:Wle}
\lambda_{i-j+1}(A)+\lambda_j(B)\leq\lambda_i(A+B),\quad j=1,\cdots,i
\end{equation}
for each $i=1,\cdots,n$.
\end{Thm}
Now, applying Theorem \ref{thm:Weyl} to $A_1$ and $A_2$, introduced above, and  choosing $j=0$ in \eqref{eq:Wue}, and $j=1$ in \eqref{eq:Wle}, we get
\begin{equation*}
\lambda_i(A_1)+\lambda_1(A_2)\leq\lambda_{i}(A_1+A_2)\leq\lambda_i(A_1),
\end{equation*}
that is, 
\begin{equation}\label{eq:con-proof-lem}
a_iu'+\sum_{j=1}^na_j^2x_j^2u''\leq \lambda_i(D^2u)\leq a_iu',\quad i=1,\cdots,n.
\end{equation}

 In particular, if $a_1=\cdots=a_n=a$, choosing $i=1$ in \eqref{eq:con-proof-lem} yields 
\begin{equation*}
 au'+\sum_{j=1}^na^2x_j^2u''\leq \lambda_1(D^2u)\leq au'.
\end{equation*} 
Then choosing $j=0$ in \eqref{eq:Wue}, and $j=2$ in \eqref{eq:Wle}, we get
\begin{equation}\label{eq:D2u-a}
au'(s)\leq \lambda_i(D^2u)\leq au'(s),\quad i\geq2.
\end{equation}
Note from \eqref{eq:D2u} that 
\begin{equation}\label{eq:trace}
\sum_{j=1}^n\lambda_j(D^2u)=\sum_{j=1}^na_ju'+\sum_{j=1}^na_j^2x_j^2u''.
\end{equation}
This, together with \eqref{eq:D2u-a}, forces that $$\lambda_1(D^2u)=au'(s)+2asu''(s),\quad \lambda_i(D^2u)=au', \quad i\geq 2.$$
\end{proof}

 By virtue of Lemma \ref{lem:D2u} and \eqref{eq:trace}, we can address 
\begin{equation}\label{eq:thetaD2u}
\lambda_i(D^2u(s))=a_iu'(s)+\theta_i\sum_{j=1}^na_j^2x_j^2u''(s),
\end{equation}
where $\theta_i\geq0$, depending on $s$ and satisfying $\sum_{i=1}^n\theta_i=1$. This enables us to get the following property of $f(\lambda(D^2u))$. Without loss of generality, we here assume that $\frac{\partial f}{\partial \lambda_1}(\lambda(A))=\max_{1\leq i\leq n}\frac{\partial f}{\partial \lambda_i}(\lambda(A))$.
\begin{Lem}\label{lem:ff}
Assume \eqref{eq:increase}, $u'(s)>0$ and $u''(s)\leq 0$. Let $\delta>0$ be fixed. If
\begin{equation*}
\lim_{s\to+\infty}u'(s)=1,\quad\lim_{s\to+\infty}su''(s)=0,
\end{equation*}
then there exists $\bar{s}=\bar{s}(\delta,f,A)>0$, such that for any $s>\bar{s}$,
\begin{equation*}
f(\lambda(D^2u))\geq f(a_1u'+(2a_n+\delta)su'', a_2u',\cdots, a_nu').
\end{equation*}
\end{Lem}

\begin{proof}
First, by \eqref{eq:thetaD2u} we see that $\lambda_i(D^2u)>0$ $(i=1,\cdots,n)$ for sufficiently large $s$, provided $u'>0$ and $su''\to0$ as $s\to+\infty$. This implies $\lambda(D^2u(s))\in\Gamma$ when $s$ is   very large. The same holds also for   
$$\bar{a}_\delta:=(a_1u'+(2a_n+\delta)su'',a_2u',\cdots, a_nu').$$ Hence, the expression of Lemma \ref{lem:ff} is valid.

To show the assertion true, we compute
\begin{align}
&f(\lambda(D^2u))-f(\bar{a}_\delta)\notag\\
=& [\theta_1U-(2a_n+\delta)s]u''\frac{\partial f}{\partial\lambda_1}(\tilde{a}_\delta)+Uu''\sum_{i=2}^n\theta_i\frac{\partial f}{\partial\lambda_i}(\tilde{a}_\delta):=\mathrm{I}\label{eq:thetaff}
\end{align}
where $$U:=\sum_{i=1}^n a_i^2x_i^2,\quad \text{and}\quad \tilde{a}_\delta:=t_0\bar{a}_\delta+(1-t_0)\lambda(D^2u)$$ for some $0\leq t_0\leq 1$. Since $$Uu''=O(su'') \quad\text{and}\quad u'\to 1,\  \text{as } s\to+\infty,$$ we have $$\tilde{a}_\delta\to \lambda(A) \quad \text{and}\quad \nabla f(\tilde{a}_\delta)\to\nabla f(\lambda(A)),\  \text{as } s\to+\infty.$$ Namely, for small $\epsilon>0$,  there exists $s(\epsilon)$ such that $$\left|\frac{\partial f}{\partial\lambda_i}(\tilde{a}_\delta)-\frac{\partial f}{\partial\lambda_i}(\lambda(A))\right|<\epsilon,\quad i=1,\cdots,n$$
when $s>s(\epsilon)$. It thus follows from \eqref{eq:thetaff} that
\begin{align*}
\mathrm{I}&\geq \frac{\partial f}{\partial\lambda_1}(\lambda(A))[U-(2a_n+\delta)s]u''+\epsilon[(1-2\theta_1)U+(2a_n+\delta)s]u''\\
&\geq -\delta su''\frac{\partial f}{\partial\lambda_1}(\lambda(A))+\epsilon(4a_n+\delta)su''.
\end{align*}
Therefore, $\mathrm{I}\geq 0$ if $\epsilon\leq \delta\frac{\partial f}{\partial\lambda_1}(\lambda(A))/(4a_n+\delta)$, which proves the conclusion.
\end{proof}

\begin{Rk}\label{rk:ff}
There are two special cases in which the fixed $\delta$ could be $0$; that is, 
\begin{equation}\label{eq:ff-0}
f(\lambda(D^2u))\geq f(a_1u'+2a_nsu'', a_2u',\cdots, a_nu').\end{equation}
The first one is when $a_1=a_2=\cdots=a_n$, \eqref{eq:ff-0} stands with equality for any $s>0$, which is obvious by Lemma \ref{lem:D2u}. The second one is if $\frac{\partial f}{\partial\lambda_1}(\lambda(A))>\frac{\partial f}{\partial\lambda_i}(\lambda(A))$, $i=2,\cdots,n$, then \eqref{eq:ff-0} stands with $s>0$ large enough. Indeed, given $\epsilon>0$, there is $s_0>0$ such that 
\begin{equation*}
\frac{\partial f}{\partial\lambda_i}(\tilde{a}_0)+\epsilon<\frac{\partial f}{\partial\lambda_1}(\tilde{a}_0)-\epsilon<\frac{\partial f}{\partial\lambda_1}(\lambda(A))\leq \frac{\partial f}{\partial\lambda_1}(\tilde{a}_0)+\epsilon, 
\end{equation*}
for $s>s_0$, $i=2,\cdots,n$. Similar to \eqref{eq:thetaff}, we obtain
\begin{align*}
&f(\lambda(D^2u))-f(a_1u'+2a_nsu'',a_2u',\cdots,a_nu')\\
\geq & \left(\frac{\partial f}{\partial\lambda_1}(\lambda(A))-\epsilon\right)(U-2a_ns)u''\geq 0
\end{align*}
only if $\epsilon<\frac{\partial f}{\partial\lambda_1}(\lambda(A))$.
\end{Rk}

 We continue our discussion by asking the diagonal matrix $A$ that appeared above, to be in the set $\mathcal{A}$ hereafter. In order to solve $f(\lambda(D^2u))\geq 1$ currently, Lemma \ref{lem:ff} inspires us to consider the implicit ordinary differential equation
\begin{equation}\label{eq:ode-f}
f\left(a_1u'+(2a_n+\delta)su'',a_2u',\ldots,a_nu'\right)=1.
\end{equation}
It is clear that if a solution of \eqref{eq:ode-f} exists and agrees with the assumption of Lemma \ref{lem:ff}, it is indeed a subsolution of \eqref{eq:pro}. For this reason, concerning \eqref{eq:ode-f} with initial data 
\begin{equation}\label{eq:ode-f-i}
u(1)=c_1,\quad u'(1)=c_2,
\end{equation}
where $c_1,c_2$ are given constants, we show the existence of its solutions on $[1,+\infty)$ and determine their asymptotic behavior at infinity. Precisely, 

\begin{Pro}\label{pro:sub}
Assume that \eqref{eq:increase}, \eqref{eq:c} and \eqref{eq:boundary} hold. Let $c_1\in\mathbb{R}$, $c_2>1$ and let $$\alpha_\delta=\frac{\sum_{i=1}^na_i\frac{\partial f}{\partial\lambda_i}(\lambda(A))}{(2a_n+\delta)\frac{\partial f}{\partial\lambda_1}(\lambda(A))}$$
be defined as in \eqref{eq:alpha-d}. For each $\delta\geq0$ such that $\alpha_\delta>1$, the ODE \eqref{eq:ode-f} with \eqref{eq:ode-f-i} has a smooth solution $u_{c_1,c_2,\delta}(s)$ on $[1,+\infty)$, such that $u_{c_1,c_2,\delta}'(s)\geq 1$ and
\begin{equation*}
u_{c_1,c_2,\delta}(s)=s+c_1+\mu(c_2)+O(s^{1-\alpha_\delta}),\quad s\rightarrow +\infty,
\end{equation*}
where $\mu(c_2)$ is a strictly increasing function of $c_2$ and satisfies
\begin{equation*}
\lim_{c_2\to+\infty}\mu(c_2)=+\infty.
\end{equation*}
\end{Pro}

The proof of Proposition \ref{pro:sub} is nontrivial. Modifying a technique used in \cite{Li-Bao-2014}, it first extracts a first-order ODE about $u'$ from \eqref{eq:ode-f}, and then get its solvability by applying Picard-Lindel\"of theorem, whereby a solution of \eqref{eq:ode-f} exists through integral. After that, the solution's asymptotic behavior and the dependence on initial values will be exploited from the extracted ODE via delicate analyses. For convenience, we shall split this detailed process into three lemmas to precisely present.

The first one, which helps us to find the first-order ODE satisfied for $u'$, is below. 
\begin{Lem}\label{lem:fgw}
Assume \eqref{eq:increase} and \eqref{eq:boundary}.  There is a unique monotone decreasing smooth function $g$ defined on $[1,+\infty)$ such that $$(g(w),a_2w,\cdots,a_nw)\in\Gamma,$$ and 
\begin{equation}\label{eq:fgw}
f(g(w),a_2w,\cdots,a_nw)=1,
\end{equation}
for $w\in[1,+\infty)$.
\end{Lem}

\begin{proof}
The assumption \eqref{eq:increase} and $A\in\mathcal{A}$ imply that
\begin{equation*}
f(a_1,a_2w,\cdots,a_nw)>1,\quad\text{if }w>1.
\end{equation*}
When $\epsilon$ is small enough, \eqref{eq:boundary} implies that $f(\epsilon,a_2w,\cdots,a_nw)<1$. Thus, by intermediate theorem and the convexity of $\Gamma$, there exists a unique $g(w)$ such that 
\begin{equation*}
f(g(w),a_2w,\cdots,a_nw)=1.
\end{equation*}
Differentiating it with respect to $w$, we find $$g'(w)\frac{\partial f}{\partial\lambda_1}(g(w),a_2w,\cdots,a_nw)+\sum_{i=2}^na_i\frac{\partial f}{\partial\lambda_i}(g(w),a_2w,\cdots,a_nw)=0.$$
So, 
\begin{equation}\label{eq:dw-g}
g'(w)=-\frac{\sum_{i=2}^na_i\frac{\partial f}{\partial\lambda_i}(g(w),a_2w,\cdots,a_nw)}{\frac{\partial f}{\partial\lambda_1}(g(w),a_2w,\cdots,a_nw)}.
\end{equation}
Accordingly, by the smoothness of $f$ and \eqref{eq:increase}, $g$ is smooth  and $g'(w)<0$.
\end{proof}

Combining \eqref{eq:ode-f} and \eqref{eq:fgw}, we replace the $u'(s)$ of \eqref{eq:ode-f} by $w(s)$ to deduce
$$g(w(s))=a_1w(s)+(2a_n+\delta)sw'(s)$$
for $w(s)\geq 1$ and $s\geq 1$. With \eqref{eq:ode-f-i}, this yields the first-order ODE: 
\begin{equation}\label{eq:ode-w-1}
\begin{cases}
\frac{dw}{ds}=\frac{g(w)-a_1w}{(2a_n+\delta)s},\\
w(1)=c_2.
\end{cases}
\end{equation}
Notice from Lemma \ref{lem:fgw} that $g(1)=a_1$ and $g(w)<a_1$ if $w>1$. This implies $\frac{dw}{ds}\leq 0$ in \eqref{eq:ode-w-1}. If $c_2=1$, \eqref{eq:ode-w-1} gives a constant solution $w(s)\equiv1$. Hence, one has to search for the nontrival solution when $c_2>1$. Actually, applying the well known Picard-Lindel\"of theorem directly, it is seen that problem \eqref{eq:ode-w-1} admits locally a unique smooth solution, denoted by $w_{c_2,\delta}(s)$, which can be extended to the whole interval $[1,+\infty)$ by the maximal existence theorem in the theory of ODEs. Also, when $c_2>1$,  we have $1<w_{c_2,\delta}(s)<c_2$ for $s>1$ due to the uniqueness of solution to \eqref{eq:ode-w-1}, and $w_{c_2,\delta}$ will converge to $1$ at infinity:
\begin{Lem}\label{lem:fgw-1}
Let $c_2>1$ and let $\delta\geq 0$ be supposed as in Proposition \ref{pro:sub}. If $w_{c_2,\delta}(s)$ is a solution of \eqref{eq:ode-w-1} on $[1,+\infty)$, then
\begin{equation*}
\lim_{s\to+\infty} w_{c_2,\delta}(s)=1,
\end{equation*}
and
\begin{equation}\label{eq:asym-s-w}
w_{c_2,\delta}(s)-1=O(s^{-\alpha_\delta}),\quad\text{as}\ s\to+\infty,
\end{equation}
where $\alpha_\delta$ is as in Proposition \ref{pro:sub}.
\end{Lem}

\begin{proof}
First, $\lim_{s\to+\infty}w_{c_2,\delta}(s)$ exists, by $w_{c_2,\delta}'(s)\leq0$ and $w_{c_2,\delta}(s)\geq1$.
Let us show this limit is $1$. Since
\begin{equation}\label{eq:ds-w-c2}
\frac{d}{ds}(w_{c_2,\delta}-1)=\frac{g(w_{c_2,\delta})-a_1-a_1w_{c_2,\delta}+a_1}{(2a_n+\delta)s}\leq -\frac{a_1w_{c_2,\delta}-a_1}{(2a_n+\delta)s},
\end{equation}
that is, $$\frac{d(w_{c_2,\delta}-1)}{w_{c_2,\delta}-1}\leq\frac{-a_1ds}{(2a_n+\delta)s}.$$
Integrating the above gives $w_{c_2,\delta}-1\leq Cs^{-\frac{a_1}{(2a_n+\delta)}}$, where $C>0$ is some constant.  This leads to $$\lim_{s\to+\infty}w_{c_2,\delta}(s)=1.$$

Next, we prove \eqref{eq:asym-s-w}. It has been known that $w_{c_2,\delta}-1\leq Cs^{-\frac{a_1}{(2a_n+\delta)}}$. Rewrite \eqref{eq:ds-w-c2} as:
\begin{align}
\frac{d}{ds}(w_{c_2,\delta}-1)&=\frac{g(w_{c_2,\delta})-g(1)-a_1w_{c_2,\delta}+a_1}{(2a_n+\delta)s}\notag\\
&=\frac{w_{c_2,\delta}-1}{(2a_n+\delta)s}\left(g'(\theta w_{c_2,\delta}+(1-\theta))-a_1\right)\label{eq:ds-w-c2-1}
\end{align}
for some $\theta\in(0,1)$. Note $1< \theta w_{c_2,\delta}+(1-\theta)<c_2$. By \eqref{eq:dw-g}, it follows that 
\begin{align}
g'(\theta w_{c_2,\delta}+(1-\theta))-a_1&=g'(\theta w_{c_2,\delta}+(1-\theta))-g'(1)+g'(1)-a_1\notag\\
&\leq M_{g'}(w_{c_2,\delta}-1)-(2a_n+\delta)\alpha_\delta\notag\\
&\leq M_{g'}\left(Cs^{-\frac{a_1}{(2a_n+\delta)}}\right)-(2a_n+\delta)\alpha_\delta,\label{eq:g-th-w-c2}
\end{align}
where $M_{g'}(s)$ denotes the modulus of continuity of $g'$, i.e. 
$$
M_{g'}(s):=\sup_{|q_1-q_2|\leq s;\ 1<q_1<q_2<c_2}|g'(q_1)-g'(q_2)|.
$$
Combining \eqref{eq:ds-w-c2-1} and \eqref{eq:g-th-w-c2}, one obtains
\begin{equation*}
\frac{d(w_{c_2,\delta}-1)}{w_{c_2,\delta}-1}\leq \frac{M_{g'}\left(Cs^{-\frac{a_1}{(2a_n+\delta)}}\right)-(2a_n+\delta)\alpha_\delta}{(2a_n+\delta)s}ds.
\end{equation*}
Hence, by integrating from $1$ to $s$,  we have 
\begin{equation}\label{eq:n-w-c2}
\ln(w_{c_2,\delta}-1)-\ln(c_2-1)\leq\int_1^s\frac{M_{g'}\left(Ct^{-\frac{a_1}{(2a_n+\delta)}}\right)}{(2a_n+\delta)t}\,dt-\alpha_\delta\ln s.
\end{equation}
Since $g'$ is Dini continuous,
$$
\int_1^\infty\frac{1}{t}M_{g'}\left(Ct^{-\frac{a_1}{(2a_n+\delta)}}\right)\,dt=\frac{(2a_n+\delta)}{a_1}\int_0^{C}\frac{M_{g'}(s)}{s}ds<\infty.
$$
Therefore, \eqref{eq:n-w-c2} indicates that $0<w_{c_2,\delta}(s)-1\leq C_1s^{-\alpha_\delta}$ for $s>1$. \eqref{eq:asym-s-w} is proved.
\end{proof}

It remains to show the dependence of the solution of \eqref{eq:ode-w-1} on the initial value $c_2$. 
\begin{Lem}\label{lem:fgw-2}
Let $w_{c_2,\delta}(s)$ be a solution of \eqref{eq:ode-w-1} with $\delta\geq 0$ and $c_2>1$. Then $0<\frac{\partial w_{c_2,\delta}}{\partial c_2}\leq 1$ and $\lim_{c_2\to+\infty}w_{c_2,\delta}(s)=+\infty$ for $s\in[1,+\infty)$.
\end{Lem}
\begin{proof}
  We write $z(s):=\frac{\partial w_{c_2,\delta}}{\partial c_2}(s)$ and infer from \eqref{eq:ode-w-1} that
\begin{equation*}
\begin{cases}
\frac{dz}{ds}=\frac{g'(w_{c_2,\delta})-a_1}{(2a_n+\delta)s}z,\\
z(1)=1.
\end{cases}
\end{equation*}
This results in $z(s)=e^{C_0}$ where 
$$
C_0:=\int_1^s\frac{g'(w_{c_2,\delta})-a_1}{(2a_n+\delta)t}\,dt. $$
Noticing $g'(w_{c_2,\delta})\leq0$ and $C_0\leq 0$, it immediately follows that $0<z\leq1$.

To show $\lim_{c_2\to+\infty}w_{c_2,\delta}(s)=+\infty$ for given $s\in[1,+\infty)$,  we notice from \eqref{eq:dw-g} that 
$$g'(w_{c_2,\delta})\geq -\sum_{i=2}^na_i.$$
So $$C_0\geq\int_1^s\frac{-\sum_{i=1}^na_i}{(2a_n+\delta)t}\,dt:=C_1(A,\delta,s)>-\infty.$$
This implies that  $z\geq e^{C_1}>0$ for each fixed $s\geq 1$. Thus the conclusion holds.
\end{proof}
With Lemmas \ref{lem:fgw}-\ref{lem:fgw-2}, we now are in a position to prove Proposition \ref{pro:sub}.
\begin{proof}[Proof of Proposition \ref{pro:sub}]
As argued after Lemma \ref{lem:fgw}, if $u$ satisfies \eqref{eq:ode-f}, then its derivative $u'$ satisfies \eqref{eq:ode-w-1}. Thus, given a solution $w_{c_2,\delta}(s)$ of  \eqref{eq:ode-w-1} with $c_2>1$, we let 
\begin{equation}\label{eq:def-u-d}
u_{c_1,c_2,\delta}(s)=\int_1^sw_{c_2,\delta}(t)\,dt+c_1.
\end{equation}
Clearly, $u_{c_1,c_2,\delta}$ is a solution of problem \eqref{eq:ode-f} and \eqref{eq:ode-f-i}, and $u_{c_1,c_2,\delta}'=w_{c_2,\delta}\geq 1$ as well. Moreover, rewrite it as 
\begin{align*}
u_{c_1,c_2,\delta}(s)&=\int_1^s(w_{c_2,\delta}(t)-1)\,dt+s-1+c_1\\
&=\int_1^\infty(w_{c_2,\delta}(t)-1)\,dt+s-1+c_1-\int_s^\infty(w_{c_2,\delta}(t)-1)\,dt.
\end{align*}
By \eqref{eq:asym-s-w}, we put $\mu(c_2):=\int_1^{\infty}(w_{c_2,\delta}(t)-1)\,dt-1<\infty$. And we get
$$u_{c_1,c_2,\delta}(s)=s+c_1+\mu(c_2)+O(s^{1-\alpha_\delta}),\quad\text{as }s\to\infty.$$

Next, we prove $\lim_{c_2\to\infty}\mu(c_2)=\infty$. Lemma \ref{lem:fgw-2} shows that $\mu(c_2)$ is strictly increasing about $c_2$. Since for $s>1$,
\begin{equation*}
\frac{d^2w_{c_2,\delta}}{ds^2}=\frac{(g'(w_{c_2,\delta})-a_1-1)(g(w_{c_2,\delta})-a_1w_{c_2,\delta})}{(2a_n+\delta)s^2}>0.
\end{equation*}
There holds
\begin{equation*}
w_{c_2,\delta}(s)>\frac{g(c_2)-a_1c_2}{2a_n+\delta}(s-1)+c_2,\quad\text{for  }s>1.
\end{equation*}
Then, by the definition of $\Gamma$, we have $$-\sum_{i=2}^n a_i q<g(q)<a_1,\quad \text{for } q>1.$$ Hence,
\begin{align*}
&\int_1^\infty(w_{c_2,\delta}(t)-1)\,dt\\
\geq &\int_1^{\frac{(1-c_2)(2a_n+\delta)}{g(c_2)-a_1c_2}+1} \left(\frac{g(c_2)-a_1c_2}{2a_n+\delta}(t-1)+c_2-1\right)dt\\
\geq &{}\frac{(2a_n+\delta)(c_2-1)^2}{2(a_1c_2-g(c_2))}\geq\frac{(2a_n+\delta)(c_2-1)^2}{2c_2\sum_{i=1}^na_i}.
\end{align*}
So that
\begin{equation*}
\lim_{c_2\to\infty}\mu(c_2)=\infty.
\end{equation*}
\end{proof}

Let us conclude this section. Proposition \ref{pro:sub} provides a solution of \eqref{eq:ode-f}, i.e. $u_{c_1,c_2,\delta}(s)$, with the property that $$u_{c_1,c_2,\delta}'=w_{c_2,\delta}\geq 1\quad \text{and}\quad u_{c_1,c_2,\delta}''=w_{c_2,\delta}'\leq 0.$$ Furthermore, $$u_{c_1,c_2,\delta}'\to 1\quad \text{and}\quad su_{c_1,c_2,\delta}''=O(s^{-\alpha_\delta})\to 0,\quad \text{as } s\to+\infty.$$ These exactly agree with the assumptions of Lemma \ref{lem:ff}. Consequently, $u_{c_1,c_2,\delta}(s)$ will be a subsolution of equation \eqref{eq:pro} when $s$ is large enough. More precisely, we arrive at the following final result of the section, which plays a key role in proving Theorem \ref{thm:main}.
\begin{Pro}\label{pro:sub-f}
Assume \eqref{eq:increase}, \eqref{eq:c} and \eqref{eq:boundary}. Let $A\in\mathcal{A}$ be diagonal  with $\lambda(A)=(a_1,a_2,\cdots,a_n)$, $s=\frac12x^TAx$ $(x\in\mathbb{R}^n)$, and let $\alpha_\delta$ be as in \eqref{eq:alpha-d}. For each $\delta>0$ satisfying $\alpha_\delta>1$, there is $\bar{s}>1$, depending on $\delta$, $A$ and $f$, such that when $s>\bar{s}$ the function $u_{c_1,c_2,\delta}(s)$ given by Proposition \ref{pro:sub} is a smooth subsolution of equation \eqref{eq:pro} and fulfills
\begin{equation*}
u_{c_1,c_2,\delta}(s)=s+c_1+\mu(c_2)+O(s^{1-\alpha_\delta}),\quad\text{as } s\to+\infty.
\end{equation*} 

In particular, if $a_1=a_2=\cdots=a_n$ or $\frac{\partial f}{\partial\lambda_1}(\lambda(A))>\frac{\partial f}{\partial\lambda_i}(\lambda(A))$, $i=2,\cdots,n$, the above assertion holds further for $\delta\geq 0$ whenever $\alpha_\delta>1$.
\end{Pro}
\begin{proof}
It follows directly by combining Lemma \ref{lem:ff}, Remark \ref{rk:ff} and Proposition \ref{pro:sub}.
\end{proof}

\section{Proof of Theorem \ref{thm:main}  and some applications}\label{sec:3}
\subsection{Proof of Theorem \ref{thm:main}} \label{subsec:3-1}
We will prove Theorem \ref{thm:main} by applying an adapted Perron's method. Roughly speaking, the viscosity solution of problem \eqref{eq:pro} prescribed by \eqref{eq:asym-s} can be obtained by splicing together the supremum of barriers at the boundary of domain and the subsolutions of \eqref{eq:pro} constructed in Proposition \ref{pro:sub-f}. Such arguments have been employed in \cite{Bao-Li-Li-2014, Li-Li-Zhao-2019, Li-Li-2018,Li2019, Caffarelli-Li-2003} for the solvability of problem \eqref{eq:pro} associated with those special forms \eqref{eq:k-sigma}-\eqref{eq:Lag}. To present the precise proof in the general setting, we need to introduce several lemmas.

First, we introduce the exact statements about the Perron's method and comparison principle. They are adaptions of those appeared in \cite{Ishii1989,Ishii1992,Jensen1988,Urbas1990}, and one may consult \cite{Li-Bao-2014} for the specific proof of them. 
\begin{Lem}\label{lem:perron-m}
Assume \eqref{eq:increase}. Let $\Omega$ be a domain in $\mathbb{R}^n$, $\varphi\in C^0(\partial\Omega)$, and $\underline{u},\bar{u}\in C^0(\overline{\Omega})$ respectively to be the viscosity subsolution and supersolution to $f(\lambda(D^2u))=1$. Suppose $\underline{u}\leq \bar{u}$, and $\underline{u}=\varphi$ on $\partial \Omega$. If $v\in C^0(\overline{\Omega})$ is also the viscosity subsolution to $f(\lambda(D^2u))=1$ with $$\underline{u}\leq v\leq\bar{u}\  \text{in}\ \Omega\ \ \text{and}\ \ v=\varphi\ \mathrm{on}\ \partial\Omega,$$
then for all such $v$, $$u(x):=\sup\{v(x)\}$$ 
is the unique viscosity solution of problem 
\begin{equation*}
\left\{
\begin{array}{ll}
f(\lambda(D^2u))=1, & \mathrm{in}\ \Omega,\\
u=\varphi, & \mathrm{on}\ \partial \Omega.
\end{array}
\right.
\end{equation*}
\end{Lem}

\begin{Lem}\label{lem:compar}
Assume \eqref{eq:increase}. Let $\Omega$ be a domain in $\mathbb{R}^n$. If $u\in\mathrm{USC}(\overline{\Omega})$, $v\in \mathrm{LSC}(\overline{\Omega})$ are respectively the viscosity subsolution and supersolution to $f(\lambda(D^2u))=1$ and $u\leq v$ on $\partial\Omega$, then $u\leq v$ in $\Omega$.
\end{Lem}

Besides, we also need the following  existence result of barrier functions on the boundary, which has proved in \cite{Caffarelli-Li-2003,Bao-Li-Li-2014}.  

\begin{Lem}\label{lem:w-xi}
Let $D$ be a bounded strictly convex domain of $\mathbb{R}^n, n\geq3$, $\partial D\in C^2$, $\varphi\in C^2(\partial D)$ and let $A$ be an invertible and symmetric matrix. There exists some constant $C$, depending only on $n, \|\varphi\|_{C^2(\partial D)}$, the upper bound of $A$, the diameter and the convexity of $D$, and the $C^2$ norm of $\partial D$, such that for every $\xi\in\partial D$, there exists $\bar{x}(\xi)\in\mathbb{R}^n$ satisfying $$|\bar{x}(\xi)|\leq C\quad\text{and}\quad w_\xi<\varphi\quad\text{on}\quad \partial D\setminus\{\xi\},$$ where $$w_\xi(x)=\varphi(\xi)+\frac12\left[(x-\bar{x}(\xi))^TA(x-\bar{x}(\xi))-(\xi-\bar{x}(\xi))^TA(\xi-\bar{x}(\xi))\right]$$
for $x\in\mathbb{R}^n$.
\end{Lem}

We now start to prove Theorem \ref{thm:main}, provided that $A\in\mathcal{A}$ is of the form $$A=\text{diag}(a_1,a_2,\cdots,a_n)$$ 
with $0<a_1\leq a_2\leq\cdots\leq a_n$, and that $b=0$.

\begin{proof}[Proof of Theorem \ref{thm:main}]
 For $s>0$, let $$E(s):=\left\{x\in\mathbb{R}^n~|~\frac12x^TAx<s\right\}.$$ Without the loss of generality, we assume $E(1)\subset D\subset E(\bar{s})$, where $\bar{s}$ is defined in Proposition \ref{pro:sub-f}. We first construct a viscosity subsolution $\underline{u}$ of \eqref{eq:pro} with $\underline{u}=\varphi$ on $\partial D$.

Set for $c_2\geq1$ and $s>1$,
\begin{equation*}
\omega_{c_2}(x)=u_{c_1,c_2,\delta}(s)-u_{c_1,c_2,\delta}(\bar{s})+\beta,
\end{equation*}
where $u_{c_1,c_2,\delta}$ is given in Proposition \ref{pro:sub-f}, and set 
$$\beta:=\min\{w_{\xi}(x)~|~\xi\in\partial D, x\in \overline{E(\bar{s})}\setminus D\}.$$
Here $w_\xi$ is introduced in Lemma \ref{lem:w-xi}. We have known, by Proposition \ref{pro:sub-f}, that $\omega_{c_2}$ is a smooth subsolution of \eqref{eq:pro} when $s>\bar{s}$ and there holds
\begin{equation}\label{eq:asym-p-u}
\omega_{c_2}(x)=\frac{1}{2}x^TAx+\mu_{\bar{s}}(c_2)+O\left(|x|^{2-2\alpha_\delta}\right),\quad \text{as } |x|\to\infty,
\end{equation}
where $$\mu_{\bar{s}}(c_2):=\int_{\bar{s}}^{\infty}(w_{c_2,\delta}(t)-1)\,dt+\beta-\bar{s},
$$ here $w_{c_2,\delta}(t)$ is the solution of \eqref{eq:ode-w-1}.
Due to the monotonicity of $u_{c_1,c_2,\delta}(s)$, 
\begin{equation}\label{eq:omega-beta}
\omega_{c_2}\leq\beta,\quad\mathrm{in}\ E(\bar{s})\setminus\overline{D}, \forall c_2>1.
\end{equation}
Define then
$$\hat{b}:=\max\{w_{\xi}(x)~|~\xi\in\partial D, x\in \overline{E(\bar{s})}\setminus D\}.$$
We next take a constant $c_*$ which will be fixed later. If $c_*>\hat{b}$, then $$\mu_{\bar{s}}(1)=\beta-\bar{s}<\beta\leq \hat{b}<c_*.$$
Also, by Lemma \ref{lem:fgw-2}, $\mu_{\bar{s}}(c_2)$ is strictly increasing and $\lim_{c_2\to\infty}\mu_{\bar{s}}(c_2)=\infty$. Thus, for every $c>c_*$, there exists a unique $\alpha(c)>1$ such that 
\begin{equation}\label{eq:mu-alpha-c}
\mu_{\bar{s}}(\alpha(c))=c.
\end{equation}

Set $$\underline{w}(x)=\max\{w_{\xi}(x)~|~\xi \in \partial D\}.$$
It is clear from Lemma \ref{lem:w-xi} that $\underline{w}$ is a locally Lipschitz function in $\mathbb{R}^n\setminus D$, and $\underline{w}=\varphi$ on $\partial D$. Since $w_\xi$ is a smooth convex solution of \eqref{eq:pro}, $\underline{w}$ is a viscosity subsolution of equation \eqref{eq:pro} in $\mathbb{R}^n\setminus\overline{D}$. By \eqref{eq:def-u-d}, we fix a number $\hat{s}>\bar{s}$, and then choose another number $\hat{c}>0$ such that $$\min_{\partial E(\hat{s})}\omega_{\hat{c}}>\max_{\partial E(\hat{s})}\underline{w}.$$

Fix $c_*$ to satisfy $c_*\geq \mu_{\bar{s}}(\hat{c})$ and $c_*>\hat{b}$. Then by \eqref{eq:mu-alpha-c}, for $c\geq c_*$, we have $\alpha(c)=\mu_{\bar{s}}^{-1}(c)\geq\mu_{\bar{s}}^{-1}(c_*)\geq\hat{c}$, and thereby 
\begin{equation}\label{eq:omega-c-w}
\omega_{\alpha(c)}\geq \omega_{\hat{c}}>\underline{w},\quad \mathrm{on}\ \partial E(\hat{s}).
\end{equation}
Since \eqref{eq:omega-beta}, we have 
\begin{equation}\label{eq:omega-beta-under-w}
\omega_{\alpha(c)}\leq \beta\leq \underline{w}, \quad \mathrm{in}\ E(\bar{s})\setminus\overline{D}.
\end{equation}

We define here, for $c>c_*$,
\begin{equation*}
\underline{u}(x)=
\begin{cases}
\underline{w}(x)\quad &x\in E(\bar{s})\backslash D,\\
\max \{\omega_{\alpha(c)}(x),\underline{w}(x)\},\quad &x\in E(\hat{s})\backslash E(\bar{s}),\\
\omega_{\alpha(c)}(x),&x\in\mathbb{R}^n\backslash E(\hat{s}).
\end{cases}
\end{equation*}
Clearly,
$$\underline{u}=\underline{w}=\varphi,\quad\mathrm{on}\ \partial D.$$
And it follows from \eqref{eq:omega-c-w} that $\underline{u}=\omega_{\alpha(c)}$ in a neighborhood of $\partial E(\hat{s})$. Therefore $\underline{u}$ is locally Lipschitz in $\mathbb{R}^n\setminus D$. Since both $\omega_{\alpha(c)}$ and $\underline{w}$ are viscosity subsolutions of \eqref{eq:pro} in $\mathbb{R}^n\setminus\overline{D}$, so is $\underline{u}$.

In order to finish the proof by virtue of Lemma \ref{lem:perron-m}, we continue to find below a viscosity supersolution $\bar{u}$ of \eqref{eq:pro} with $\underline{u}\leq\bar{u}$ on $\mathbb{R}^n\setminus D$.

For $c>c_*$, define $$\bar{u}(x):=\frac12x^TAx+c,$$ which is a smooth convex solution of \eqref{eq:pro}. By  \eqref{eq:omega-beta-under-w}, we have
$$\omega_{\alpha(c)}\leq\beta\leq\hat{b}<c^*<\bar{u}, \quad\mathrm{on}\ \partial{D}.$$
Also, by \eqref{eq:asym-p-u} and \eqref{eq:mu-alpha-c}, 
$$\lim_{|x|\to\infty}\omega_{\alpha(c)}-\bar{u}=0.$$
Hence, applying Lemma \ref{lem:compar}, we deduce 
\begin{equation*}
\omega_{\alpha(c)}\leq \bar{u}, \quad\mathrm{on}\ \mathbb{R}^n\setminus D.
\end{equation*}
Then from \eqref{eq:omega-c-w} and the above, one has, for $c>c^*$,
$$w_\xi\leq\bar{u}, \quad\mathrm{on}\ \partial(E(\hat{s})\setminus D),\forall\xi\in\partial D.$$
Using Lemma \ref{lem:compar} again, we obtain
$$w_\xi\leq\bar{u},\quad\mathrm{in}\ E(\hat{s})\setminus\overline{D},\forall\xi\in\partial D,$$
and $$\underline{w}\leq\bar{u},\quad\mathrm{in}\ E(\hat{s})\setminus\overline{D}.$$
Therefore,
$$\underline{u}\leq\bar{u},\quad\mathrm{in}\ \mathbb{R}^n\setminus D. $$

Now, we conclude our proof in terms of $\underline{u}$ and $\bar{u}$. For any $c>c^*$, let $\mathcal{S}_\alpha$ denote the set containing such $v\in C^0(\mathbb{R}^n\setminus D)$ that is the viscosity subsolution of \eqref{eq:pro} in $\mathbb{R}^n\setminus\overline{D}$ satisfying 
\begin{equation*}
v=\varphi\quad\mathrm{on}\ \partial D
\end{equation*}
and 
\begin{equation*}
\underline{u}\leq v\leq \bar{u}\quad\mathrm{in} \ \mathbb{R}^n\setminus D.
\end{equation*}
Apparently, $\underline{u}\in\mathcal{S}_\alpha$. Let
$$u(x):=\sup\{v(x)|v\in\mathcal{S}_\alpha\}, \ x\in \mathbb{R}^n\setminus D.$$
Then Lemma \ref{lem:perron-m} shows that $u(x)\in C^0(\mathbb{R}^n\setminus D)$ is the unique viscosity solution of problem \eqref{eq:pro}. Moreover, since
\begin{equation*}
u(x)\geq\underline{u}=\omega_{\alpha(c)}(x)=\frac12x^TAx+c+O(|x|^{2-2\alpha_\delta}),\quad\mathrm{as}\quad x\to\infty,
\end{equation*}
and 
$$u(x)\leq\bar{u}(x)=\frac12x^TAx+c,$$
we deduce that \eqref{eq:asym-s} holds, thus completing the proof.
\end{proof}

\subsection{Some applications of Theorem \ref{thm:main}}\label{subsec:3-2}
Applying Theorem \ref{thm:main}, we can present some new results of the exterior Dirichlet problem for $k$-Hessian equations, Hessian quotient equations, and the special Lagrangian equations. This is due to the prescribed setting about the asymptotics \eqref{eq:asym-s} is different from the existing ones set separately in \cite{Bao-Li-Li-2014,Li-Li-Zhao-2019,Li2019} and even in \cite{Caffarelli-Li-2003} for the Monge-Amp\`ere equation. In fact, the difference lies in the order of tending to a quadratic polynomial at infinity, and ours, generally,  is a little smaller. We shall give corresponding examples of elements in $\mathcal{A}$ for these particular equations to make a comparison with previous related results.


For simplicity, we write $\lambda:=(\lambda_1,\lambda_2,\cdots,\lambda_n)$, omitting its relevance about some real $n\times n$ symmetric matrix $A$, to denote $\lambda(A)$.

\begin{Ex}[$k$-Hessian equations]\label{rk:sigma}
 When $f$ is \eqref{eq:k-sigma}, the $k$-th elementary symmetric function $\sigma_k(\lambda)$, $1\leq k\leq n$, we define
\begin{equation*}
\Gamma=\Gamma_k:=\{x\in\mathbb{R}^n|\sigma_j(x)>0, 1\leq j\leq k\}.
\end{equation*}
Then take $\mathcal{A}$ correspondingly as 
$$\mathcal{A}_k:=\left\{A\in S^+(n):\sigma_k(\lambda)=1,\  \frac{k}{2\lambda_n\sigma_{k-1;1}(\lambda)}>1\right\}$$
where $\sigma_{k-1;i}(\lambda)=\sigma_{k-1}(\lambda)\big{|}_{\lambda_i=0}$.
For example, let 
\begin{equation*}
H_\varepsilon=\mathrm{diag}\left(\sqrt{\frac{1}{3}}-\frac{2\varepsilon}{2+\sqrt{3}\varepsilon},\sqrt{\frac{1}{3}},\sqrt{\frac{1}{3}}+\varepsilon\right).
\end{equation*}
One can compute directly that $H_{\varepsilon}\in\mathcal{A}_{k}$ for any $\varepsilon\in[0,\frac{-3\sqrt{3}+\sqrt{39}}{6})$, $k=2$ and $n=3$.

Theorem \ref{thm:main} asserts that for each $A\in\mathcal{A}_{k}$, there is a viscosity solution to the exterior problem:
\begin{equation}\label{eq:pro-sigma}
\left\{
\begin{array}{ll}
\sigma_k(\lambda(D^2u))=1, & \mathrm{in}\ \mathbb{R}^n\setminus \overline{D},\\
u=\varphi, & \mathrm{on}\ \partial D,\\
\end{array}
\right.
\end{equation}
and
\begin{equation}\label{eq:asym-sigma}
\limsup_{|x|\to\infty}\left(|x|^{\frac{2k}{(2\lambda_n+\delta)\sigma_{k-1;1}(\lambda)}-2}\Big{|}u(x)-(\frac12x^TAx+b\cdot x+c)\Big{|}\right)<\infty,
\end{equation}
if $0<\delta<\frac{k-2\lambda_n\sigma_{k-1;1}(\lambda)}{\sigma_{k-1;1}(\lambda)}$. By contrast, when $2\leq k\leq n$, Bao-Li-Li \cite{Bao-Li-Li-2014} dealed with \eqref{eq:pro-sigma} and \eqref{eq:asym-sigma} where 
the growth function $|x|^{\frac{2k}{(2\lambda_n+\delta)\sigma_{k-1;1}(\lambda)}-2}$ of \eqref{eq:asym-sigma} is replaced by $|x|^{\frac{k}{\lambda_n\sigma_{k-1;n}(\lambda)}-2}$. The powers of these two growth functions are different. An illustration by exemplifying $H_\varepsilon$ $(\varepsilon=0.0874)$ and setting $\delta=0.1$ is: 
\begin{gather*}
\frac{4}{(2\lambda_3(H_\varepsilon)+0.1)\sigma_{1;1}(\lambda(H_\varepsilon))}-2\approx0.2528,\\
\frac{2}{\lambda_3(H_{\varepsilon})\sigma_{1;3}(\lambda(H_\varepsilon))}-2\approx0.8024.
\end{gather*}

Especially, when $k=n$, \eqref{eq:pro-sigma} reduces to the the Monge-Amp\`ere equation:
 \begin{equation}\label{eq:pro-Monge}
\left\{
\begin{array}{ll}
\det(D^2u)=1, & \mathrm{in}\ \mathbb{R}^n\setminus \overline{D},\\
u=\varphi, & \mathrm{on}\ \partial D.\\
\end{array}
\right.
\end{equation}
Our theorem proves the solvability of \eqref{eq:pro-Monge} with
\begin{equation}\label{eq:asym-Monge}
\limsup_{|x|\to\infty}\left(|x|^{\frac{2n\lambda_1}{2\lambda_n+\delta}-2}\Big{|}u(x)-(\frac12x^TAx+b\cdot x+c)\Big{|}\right)<\infty,
\end{equation}
for $0<\delta<n\lambda_1-2\lambda_n$ and the $A$ that is in
$$\left\{A\in S^+(n):\det(\lambda)=1,\  \frac{n\lambda_1}{2\lambda_n}>1\right\}:=\mathcal{A}_{m}.$$
While Caffarelli-Li \cite{Caffarelli-Li-2003} treated the problem \eqref{eq:pro-Monge} with \eqref{eq:C-Li} instead of \eqref{eq:asym-Monge}. Below are examples:
$$M_\varepsilon=\mathrm{diag}\left(1-\frac{\varepsilon}{1+\varepsilon},1,1+\varepsilon\right).$$
$M_\varepsilon\in\mathcal{A}_m$ for every $0\leq \varepsilon<\frac{\sqrt{6}-2}{2}$, $n=3$. Letting $\varepsilon=0.1$ and $\delta=0.1$, we get $n-2=1$ but
\begin{gather*}
\frac{6\lambda_1(M_{\varepsilon})}{2\lambda_3(M_{\varepsilon})+0.1}-2\approx0.3715.
\end{gather*}
\end{Ex}

\begin{Ex}[Hessian quotient equations]\label{rk:quo}
When $f$ is \eqref{eq:quotient}, the quotient of elementary symmetric function $\frac{\sigma_k}{\sigma_l}(\lambda)$, $1\leq l<k\leq n$, we still set $\Gamma=\Gamma_k$. The $\mathcal{A}$ will become
$$\mathcal{A}_q:=\left\{A\in S^+(n):\frac{\sigma_k}{\sigma_l}(\lambda)=1,\  \frac{(k-l)\sigma_l(\lambda)}{2\lambda_nH(k,l)(\lambda)}>1\right\},$$
where $$H(k,l)(
\lambda):=\max_{1\leq i\leq n}(\sigma_{k-1;i}-\sigma_{l-1;i})(\lambda).$$
Examples are
\begin{equation*}
Q_\epsilon=\mathrm{diag}\left(3-\frac{2\epsilon}{2\epsilon+3},3,3+\epsilon\right).
\end{equation*}
One can verify that $Q_\epsilon\in\mathcal{A}_q$ for $0\leq \epsilon<\frac{3\sqrt{3}-3}{4}$, $n=3$, $k=3$ and $l=2$.

Then Theorem \ref{thm:main} shows that for each $A\in\mathcal{A}_q$, there is a viscosity solution to the exterior problem:
\begin{equation}\label{eq:pro-quo}
\left\{
\begin{array}{ll}
\frac{\sigma_k}{\sigma_l}(\lambda(D^2u))=1, & \mathrm{in}\ \mathbb{R}^n\setminus \overline{D},\\
u=\varphi, & \mathrm{on}\ \partial D,\\
\end{array}
\right.
\end{equation}
and
\begin{equation}\label{eq:asym-quo}
\limsup_{|x|\to\infty}\left(|x|^{2\alpha_\delta-2}\Big{|}u(x)-(\frac12x^TAx+b\cdot x+c)\Big{|}\right)<\infty,
\end{equation}
where $\alpha_\delta=\frac{(k-l)\sigma_l(\lambda)}{(2\lambda_n+\delta)H(k,l)(\lambda)}$ for those $\delta>0$ such that $\alpha_\delta>1$. In previous work \cite{Li-Li-Zhao-2019} (see also \cite{Li-Li-2018}), the authors solved \eqref{eq:pro-quo} with \eqref{eq:asym-quo} under the growth assumption $|x|^{\theta(k,l)}$, rather than $|x|^{2\alpha_\delta-2}$. Here $$\theta(k,l):=\frac{(k-l)\sigma_l(\lambda)}{\lambda_n\sigma_{k-1;n}-\lambda_1\sigma_{l-1;1}}-2.$$
Exemplifying $Q_\epsilon$ $(\epsilon=0.1)$ and setting $\delta=0.1$, we can observe their difference:
\begin{gather*}
\frac{2\sigma_2(\lambda(Q_\epsilon))}{(2\lambda_3+0.1)(\sigma_{2;1}-\sigma_{1;1})(\lambda(Q_\epsilon))}-2\approx0.7102\\
\frac{\sigma_2(\lambda(Q_\epsilon))}{\lambda_3\sigma_{2;3}(\lambda(Q_\epsilon))-\lambda_1\sigma_{1;1}(\lambda(Q_\epsilon))}-2\approx0.8956.
\end{gather*}
\end{Ex}

\begin{Ex}[The special Lagrangian equations]\label{rk:Lag}
When $f$ involves \eqref{eq:Lag}, the special Lagrangian operator $$\frac{1}{\Theta}\sum_{i=1}^n\arctan\lambda_i,$$ we set $(n-1)\pi/2\leq\Theta<n\pi/2$ and $\Gamma=\Gamma_n$ $(\Gamma_{n-1})$ if $n$ is odd (even). The $\mathcal{A}$ will become
$$\mathcal{A}_l:=\left\{A\in S^+(n):\frac{1}{\Theta}\sum_{i=1}^n\arctan\lambda_i=1,\  \frac{1+\lambda_1^2}{2\lambda_n}\sum_{i=1}^n\frac{\lambda_i}{1+\lambda_i^2}>1\right\}.$$
Its examples are
\begin{equation*}
L_\varepsilon=\mathrm{diag}\left(\tan\left(\frac{\pi}{3}-\varepsilon\right),\tan\left(\frac{\pi}{3}\right),\tan\left(\frac{\pi}{3}+\varepsilon\right)\right),
\end{equation*}
if $0\leq \varepsilon<0.071$, $n=3$ and $\Theta=\pi$.

Theorem \ref{thm:main} proves that for each $A\in\mathcal{A}_l$, there is a viscosity solution to the exterior problem:
\begin{equation}\label{eq:pro-Lag}
\left\{
\begin{array}{ll}
\frac{1}{\Theta}\sum_{i=1}^n\arctan(\lambda_i(D^2u))=1, & \mathrm{in}\ \mathbb{R}^n\setminus \overline{D},\\
u=\varphi, & \mathrm{on}\ \partial D,\\
\end{array}
\right.
\end{equation}
and
\begin{equation}\label{eq:asym-Lag}
\limsup_{|x|\to\infty}\left(|x|^{2\alpha_\delta-2}\Big{|}u(x)-(\frac12x^TAx+b\cdot x+c)\Big{|}\right)<\infty,
\end{equation}
where $\alpha_\delta=\frac{1+\lambda_1^2}{2\lambda_n+\delta}\sum_{i=1}^n\frac{\lambda_i}{1+\lambda_i^2}$ for those $\delta>0$ such that $\alpha_\delta>1$. In \cite{Li2019}, Li considered \eqref{eq:pro-Lag} with \eqref{eq:asym-Lag} where the growth function $|x|^{2\alpha_\delta-2}$ is replaced by $|x|^{m-2}$, $$m:=\frac{\sum_{k=0}^nkc_k(\Theta)\sigma_k(\lambda)}{\sum_{k=0}^n\xi_k(\Theta,\lambda)c_k(\Theta)\sigma_k(\lambda)}$$
where $\xi_{k}$ and $c_k$ are two quantities related to $\Theta$, whose precise definition can be found in \cite{Li2019}. 
For $L_\varepsilon$ $(\varepsilon=0.035)$ and $\delta=0.001$, $2\alpha_{0.001}-2\approx0.4537$, while $m-2\approx0.8856$.
\end{Ex}

In the above examples, we have declared the new results of Theorem \ref{thm:main} for problems \eqref{eq:pro-sigma}, \eqref{eq:pro-Monge}, \eqref{eq:pro-quo} and \eqref{eq:pro-Lag}, respectively. Focusing on the individual case $A=a^*I$, we remark that our theorem generalizes the existing results to be valid for a family of prescribed asymptotic settings, not just for \eqref{eq:C-Li}.

\begin{Rk}
When $f$ takes \eqref{eq:k-sigma}-\eqref{eq:Lag} respectively, it is easy to   derive that the corresponding $a^*$ defined by \eqref{eq:c} is $\left(C_n^k\right)^{-\frac{1}{k}}$ for \eqref{eq:k-sigma}, $\left(C_n^l/C_n^k\right)^{\frac{1}{k-l}}$ for \eqref{eq:quotient}, and $\tan(\Theta/n)$ for \eqref{eq:Lag}, where $C_n^i=\frac{n!}{(n-i)!i!}$ $(i=k,l)$ are binomial coefficients. If we select accordingly the $A$ in the prescribed behaviors, imposed in \cite{Bao-Li-Li-2014, Li-Li-Zhao-2019, Li-Li-2018, Li2019} for problems \eqref{eq:pro-sigma}, \eqref{eq:pro-quo} and \eqref{eq:pro-Lag}, to be the specific $a^*I$, then all of those behaviors will exactly be \eqref{eq:C-Li}, which corresponds to Theorem \ref{thm:main}'s \eqref{eq:asym-s} with $\delta=0$. This means, in the special situation $A=a^*I$, that our theorem not only covers the existence results asserted in \cite{Bao-Li-Li-2014, Li-Li-Zhao-2019, Li-Li-2018,Li2019, Caffarelli-Li-2003} but also extend them from $\delta=0$ to any $0\leq\delta<(n-2)a^*$.

\end{Rk}

\end{document}